\documentclass[review]{elsarticle}

\usepackage{lineno,hyperref}
\usepackage{amssymb}
\modulolinenumbers[5]
\newtheorem{theorem}{Theorem}
\newtheorem{lemma}{Lemma}
\newtheorem{prop}{Proposition}

\newtheorem{rem}{Remark}
\newproof{proof}{Proof}
\journal{Journal of \LaTeX\ Templates}

\begin{document}

\begin{frontmatter}

\title{Orthogonal foliations on riemannian manifolds}


\author[mymainaddress]{Andr\'e de Oliveira Gomes\corref{mycorrespondingauthor}}
\cortext[mycorrespondingauthor]{Corresponding author}
\ead{gomes@ime.usp.br}

\author[mysecondaryaddress]{Eur\'{\i}pedes Carvalho da Silva}

\address[mymainaddress]{Departmento de Matem\'atica, Universidade de S\~ao Paulo, 05508-090 S\~ao Paulo, SP, Brazil}
\address[mysecondaryaddress]{Instituto Federal de Ci\^encia e Tecnologia do Cear\'a, 61.939-140 Maracana\'u, CE, Brazil }

\begin{abstract}
In this work, we find an equation that relates the Ricci curvature of a riemannian manifold $M$ and the second fundamental forms of two orthogonal foliations of complementary dimensions, $\mathcal{F}$ and $\mathcal{F}^{\bot}$, defined on $M$. Using this equation, we show a sufficient condition for the manifold M to be locally a riemannian product of the leaves of $\mathcal{F}$ and $\mathcal{F}^{\bot}$, if one of the foliations is totally umbilical. We also prove an integral formula for such foliations.
\end{abstract}


\end{frontmatter}

\section{Introduction}
A great motivation for this work is a result of Brito and Walczak \cite{Brito1985410} about a pair of two orthogonal foliations of complementary dimensions, $\mathcal{F}$ and $\mathcal{F}^{\bot}$, defined on a complete Riemannian manifold $M$. These authors shown that if the foliation $\mathcal{F}^{\bot}$ is \textit{totally geodesic} (i.e. each leaf of $\mathcal{F}^{\bot}$ is a totally geodesic submanifold of $M$) and the Ricci curvature of the ambient manifold $M$ is not negative, then $M$ is a locally riemannian product of the leaves of $\mathcal{F}$ and $\mathcal{F}^{\bot}$. This generalizes an analogous result proved by Abe \cite{Abe1973425} using the additional hypothesis of local symmetry of the ambient space $M$. The Brito-Walczak's result was proved using an equation that relates the Ricci curvature of the riemannian manifold $M$ and the second fundamental form of the foliation $\mathcal{F}$.
\par In this work, we find a general equation that relates the Ricci curvature of a riemannian manifold $M$ and the second fundamental forms of two orthogonal foliations of complementary dimensions, $\mathcal{F}$ and $\mathcal{F}^{\bot}$, defined on M. Using this equation, we generalize a theorem of Almeida, Brito and Colares \cite{Almeida2017573} about such foliations in which all leaves of $\mathcal{F}^{\bot}$ are totally umbilical submanifolds of $M$ (the foliation $\mathcal{F}^{\bot}$ is then said a \textit{totally umbilical} foliation). Using the same equation, we also find a more general integral formula than that obtained by Walczak \cite{Walczak1990}.

\par The structure of this paper is as follows. In section 2 we present the main definitions and notations used in rest of the paper. In section 3 we prove the general equation cited above. In section 4 we analyze the case when the foliation $\mathcal{F}^{\bot}$ is totally umbilical and we prove a generalization of the theorem 1 of \cite{Almeida2017573}. Finally, in section 5 we establish the general integral formula.

\section{The basic terminology}

\par Let $M$ be a riemannian manifold of dimension $n+p$. The symbols $\left\langle \ , \ \right\rangle$, $\nabla$ and $\mathrm{R}$ denote, respectively, the riemannian metric, the riemannian connexion and the curvature tensor of $M$. Let $\mathcal{F}$ be a $\mathcal{C}^{\infty}$-foliation of codimension $p$ on the manifold $M$ and let $\mathcal{F}^{\bot}$ be a $\mathcal{C}^{\infty}$-foliation of codimension $n$ on the manifold $M$ and  orthogonal to $\mathcal{F}$. Let $x\in M$ and $\left\{e_{1},\ldots,e_{n},e_{n+1},\ldots,e_{n+p}\right\}$ be a \textit{orthonormal adapted frame} (i.e. $e_{1},\ldots,e_{n}$ are tangent to $\mathcal{F}$ and $e_{n+1},\ldots,e_{n+p}$ are tangent to $\mathcal{F}^{\bot}$) in a neighborhood of $x$. We shall make use of the following convention on the range of indices:
\begin{eqnarray*}
	1\leq A,B,\ldots\!\!\! &\leq &\!\!\! n+p\\
	1\leq i,j,\ldots\!\!\! &\leq & \!\!\! n\\
	n+1 \leq \alpha,\beta,\ldots\!\!\! &\leq &\!\!\! n+p
\end{eqnarray*} 

\par We define the \textit{second fundamental form of $\mathcal{F}$ in the direction of $e_{\alpha}$} by
\begin{eqnarray*}
	H_{_\mathcal{F}}^{\alpha}(e_{i},e_{j})=\left\langle -\nabla_{e_{i}}e_{\alpha},e_j\right\rangle
\end{eqnarray*}
and we define the \textit{second fundamental form of $\mathcal{F^{\bot}}$ in the direction of $e_{i}$} by
\begin{eqnarray*}
	H_{\mathcal{F^{\bot}}}^{i}(e_{\alpha},e_{\beta})=\left\langle -\nabla_{e_{\alpha}}e_{i},e_{\beta}\right\rangle
\end{eqnarray*}

\par Let $X$ be a smooth vector field defined on the manifold $M$. We denote by $X^{\bot}$ and $X^{\top}$, respectively, the following smooth vector fields
\begin{eqnarray*}
	X^{\bot}\!\!\!&=&\!\!\!\sum\limits_{\alpha}\left\langle X,e_{\alpha}\right\rangle e_{\alpha}\\
	X^{\top}\!\!\!&=&\!\!\!\sum\limits_{i}\left\langle X,e_{i}\right\rangle e_{i}
\end{eqnarray*}

\par The \textit{Weingarten operators} of $H_{\mathcal{F}}^{\alpha}$ and $H_{\mathcal{F}^{\bot}}^{i}$ are given, respectively,  by
\begin{eqnarray*}
	A_{e_{\alpha}}(e_{i})=-\left(\nabla_{e_{i}}e_{\alpha}\right)^{\top}\ \ \ \ \mathrm{and} \ \ \ \ A_{e_{i}}(e_{\alpha})=-\left(\nabla_{e_{\alpha}}e_{i}\right)^{\bot}
\end{eqnarray*}
\par We define the \textit{norm of the second fundamental form $H_{\mathcal{F}}^{\alpha}$} by
\begin{eqnarray*}
	\left\|H_{\mathcal{F}}^{\alpha}\right\|=\left(\sum\limits_{i,j}\left\langle -\nabla_{e_{i}}e_{\alpha},e_j\right\rangle^{2}\right)^{1/2}
\end{eqnarray*}
and, analogously, we define the \textit{norm of the second fundamental form $H_{\mathcal{F^{\bot}}}^{i}$} by
\begin{eqnarray*}
	\left\|H_{\mathcal{F}^{\bot}}^{i}\right\|=\left(\sum\limits_{\alpha,\beta}\left\langle -\nabla_{e_{\alpha}}e_{i},e_{\beta}\right\rangle^{2}\right)^{1/2}
\end{eqnarray*}

\par The \textit{mean curvature vector of $\mathcal{F}$} is defined by
\begin{eqnarray*}
	h=\sum\limits_{i}\left(\nabla_{e_{i}}e_{i}\right)^{\bot}
\end{eqnarray*}
and the \textit{mean curvature vector of $\mathcal{F}^{\bot}$} is defined by
\begin{eqnarray*}
	h^{\bot}=\sum\limits_{\alpha}\left(\nabla_{e_{\alpha}}e_{\alpha}\right)^{\top}
\end{eqnarray*}

\par For each fixed $\alpha$, we denote by $(K^{\alpha}_{ij})$ the $n\times n$ matrix with entries given by $R(e_{\alpha},e_{i},e_{j},e_{\alpha})$. The \textit{trace} of the matrix $(K_{ij}^{\alpha})$ is then given by
\begin{eqnarray*}
	\mathrm{Tr}(K^{\alpha}_{ij})=\sum\limits_{i}R(e_{\alpha},e_{i},e_{i},e_{\alpha})
\end{eqnarray*}

\par For all smooth vector field $X$ defined on $M$ we have the following definitions
\begin{eqnarray*}
	\mathrm{div}_{\mathcal{F}}(X)\!\!&=&\!\!\sum\limits_{i}\left\langle e_{i}, \nabla_{e_{i}}X\right\rangle\\
	\mathrm{div}_{\mathcal{F^{\bot}}}(X)\!\!&=&\!\!\!\sum\limits_{\alpha}\left\langle e_{\alpha}, \nabla_{e_{\alpha}}X\right\rangle
\end{eqnarray*}

\section{The main equation}

\begin{theorem}
Let $M$ be a riemannian manifold and denote by $\mathcal{F}$ and $\mathcal{F}^{\bot}$ two orthogonal foliations of complementary dimensions on $M$. Then we have
\begin{eqnarray*}
e_{\alpha}\left\langle h,e_{\alpha}\right\rangle-\left\|H_{_\mathcal{F}}^{^\alpha}\right\|^2-\mathrm{Tr}(K^{\alpha}_{ij})=\sum\limits_{i=1}^{n}H_{\mathcal{F^{\bot}}}^{^i}(e_{\alpha},\nabla^{\bot}_{e_{i}}e_{\alpha}-[e_{\alpha},e_{i}]^{\bot})-\mathrm{div}_{\mathcal{F}}(\nabla_{e_{\alpha}}e_{\alpha})
\end{eqnarray*}
\end{theorem}

\begin{proof}
Let us take $\alpha$. We observe that
\begin{eqnarray}
\nonumber e_{\alpha}\langle h,e_{\alpha}\rangle-\left\|H_{\mathcal{F}}^{\alpha}\right\|^2-\sum_{i=1}^{n}{R(e_{\alpha},e_i,e_i,e_{\alpha})}\\
\nonumber & \hspace{-12cm} =  & \hspace{-6cm} e_{\alpha}\langle \sum_{i=1}^{n}{\nabla_{e_i}e_i},e_{\alpha}\rangle-\left\|H_{\mathcal{F}}^{\alpha}\right\|^2-\sum_{i=1}^{n}{R(e_{\alpha},e_i,e_i,e_{\alpha})}\\
\nonumber & \hspace{-12cm} =  & \hspace{-6cm}\sum_{i=1}^{n}{(\langle\nabla_{e_{\alpha}}\nabla_{e_i}e_i,e_{\alpha}\rangle + \langle \nabla_{e_i}e_i,\nabla_{e_{\alpha}}e_{\alpha}\rangle)}-\left\|H_{\mathcal{F}}^{\alpha}\right\|^2-\sum_{i=1}^{n}{R(e_{\alpha},e_i,e_i,e_{\alpha})}\\
\nonumber & \hspace{-12cm} = & \hspace{-6cm}\sum_{i=1}^{n}{(\langle\nabla_{e_{\alpha}}\nabla_{e_i}e_i,e_{\alpha}\rangle 
  + \langle \nabla_{e_i}e_i,\nabla_{e_{\alpha}}e_{\alpha}\rangle)}
	-\left\|H_{\mathcal{F}}^{\alpha}\right\|^2\\
\nonumber	& \hspace{-12cm} - & \hspace{-6cm} \sum_{i=1}^{n}{(\langle\nabla_{e_{\alpha}}\nabla_{e_i}e_i}
	-\nabla_{e_i}\nabla_{e_{\alpha}}e_i-\nabla_{[e_{\alpha},e_i]}e_i,e_{\alpha}\rangle) \\
\nonumber & \hspace{-12cm} = & \hspace{-6cm}\sum_{i=1}^{n}{\langle \nabla_{e_i}e_i,\nabla_{e_{\alpha}}e_{\alpha}\rangle}-\left\|H_{\mathcal{F}}^{\alpha}\right\|^2+\sum_{i=1}^{n}\langle\nabla_{e_i}\nabla_{e_{\alpha}}e_i,e_{\alpha}\rangle\\
& \hspace{-12cm} + & \hspace{-6cm} \sum_{i=1}^{n}{\langle\nabla_{[e_{\alpha},e_i]}e_i,e_{\alpha}\rangle}
\end{eqnarray}
On the other hand
\begin{equation}
\nabla_{e_{\alpha}}e_{i}=\sum_{j=1}^n{b_{ij}e_j}+\sum_{\beta}{b_{i\beta}e_{\beta}},
\label{eqthree} 
\end{equation}
where $b_{iA}=\left\langle \nabla_{e_{\alpha}}e_i,e_A\right\rangle$. We must have $b_{ii}=0$, since 
\begin{eqnarray*}
	0=e_{\alpha}\left\langle e_i,e_i\right\rangle=2\left\langle \nabla_{e_{\alpha}}e_i,e_i\right\rangle=2b_{ii} 
\end{eqnarray*}
We assume that the basis $\left\{e_{1},\ldots,e_{n}\right\}$ diagonalizes the operator $H_{\mathcal{F}}^{^\alpha}$. Then
\begin{equation}
\nabla_{e_i}e_{\alpha}=-\lambda_i^{\alpha}e_i+\sum_{\beta=n+1}^{n+p}{c_{\beta}e_{\beta}}.
\label{eqfour}
\end{equation}
where $c_{B}=\left\langle \nabla_{e_{i}}e_{\alpha},e_{B}\right\rangle$, and we have 
\begin{eqnarray}
	\left\langle \nabla_{e_i}e_{\alpha},e_j\right\rangle=-H_{\mathcal{F}}^{^\alpha}(e_{i},e_{j})=-\lambda_i^{\alpha}\delta_{ij}
\end{eqnarray}
Using $(2)$, $(3)$ and $(4)$
\begin{eqnarray*}
\left\langle \nabla_{[e_{\alpha},e_{i}]}e_i,e_{\alpha}\right\rangle & = & \left\langle \nabla_{\nabla_{e_{\alpha}}e_i}e_i,e_{\alpha}\right\rangle-\left\langle \nabla_{\nabla_{e_{i}}e_{\alpha}}e_i,e_{\alpha}\right\rangle\\
& = & \left\langle \nabla_{\sum\limits_j{b_{ij}e_j}+\sum\limits_\beta{b_{i\beta}e_{\beta}}}e_i,e_{\alpha}\right\rangle-\left\langle \nabla_{-\lambda_{i}^{\alpha}e_{i}+\sum\limits_\beta{c_{\beta}e_{\beta}}}e_i,e_{\alpha}\right\rangle\\
& = & \sum_j{b_{ij}\left\langle \nabla_{e_j}e_i,e_{\alpha}\right\rangle}+\sum_\beta{b_{i\beta}\left\langle \nabla_{e_{\beta}}e_i,e_{\alpha}\right\rangle}\\
& + & \lambda_{i}^{\alpha}\left\langle \nabla_{e_i}e_i,e_{\alpha}\right\rangle-\sum_{\beta}{c_{\beta}\left\langle \nabla_{e_{\beta}}e_i,e_{\alpha}\right\rangle}\\
& = & \sum_j{b_{ij}\lambda_{j}^{\alpha}\delta_{ij}}+\sum_{\beta}{(b_{i\beta}-c_{\beta})\left\langle \nabla_{e_{\beta}}e_i,e_{\alpha}\right\rangle}+(\lambda_{i}^{\alpha})^2.
\end{eqnarray*}
We now observe that
\begin{eqnarray*}
	\sum_j{b_{ij}\lambda_{j}^{\alpha}\delta_{ij}}\!\! & = &\!\! 0 \ \ \ \  (\mathrm{because}\ \  b_{ii}=0)
	\\
	b_{i\beta}-c_{\beta}\!\! & = & \!\!\left\langle [e_{\alpha},e_i],e_{\beta}\right\rangle
	\\
	H^{^i}_{\mathcal{F}^{\bot}}(e_{\beta},e_{\alpha})\!\!& = &\!\! -\left\langle \nabla_{e_{\beta}}e_i,e_{\alpha}\right\rangle
\end{eqnarray*}
and then
\begin{eqnarray}
\nonumber\left\langle \nabla_{[e_{\alpha},e_{i}]}e_i,e_{\alpha}\right\rangle & = & -\sum_{\beta}{\left\langle [e_{\alpha},e_i],e_{\beta}\right\rangle H^{^i}_{\mathcal{F}^{\bot}}(e_{\beta},e_{\alpha})}+(\lambda_{i}^{\alpha})^2\\
\nonumber& = & -\sum_{\beta}{H^{^i}_{\mathcal{F}^{\bot}}(\left\langle [e_{\alpha},e_i],e_{\beta}\right\rangle e_{\beta},e_{\alpha})}+(\lambda_{i}^{\alpha})^2\\
\nonumber& = & -H^{^i}_{\mathcal{F}^{\bot}}(\sum_{\beta}{\left\langle [e_{\alpha},e_i],e_{\beta}\right\rangle e_{\beta}},e_{\alpha})+(\lambda_{i}^{\alpha})^2\\
& = & -H^{^i}_{\mathcal{F}^{\bot}}([e_{\alpha},e_{i}]^{\bot},e_{\alpha})+(\lambda_{i}^{\alpha})^2
\end{eqnarray}
By (1) and (5) we have
\begin{eqnarray}
	\nonumber e_{\alpha}\left\langle h,e_{\alpha}\right\rangle-\left\|H_{\mathcal{F}}^{\alpha}\right\|^2-\sum_{i=1}^{n}{R(e_{\alpha},e_i,e_i,e_{\alpha})}\\
	& \hspace{-12cm} = & \hspace{-6cm}\sum_{i=1}^n{\left\langle \nabla_{e_i}e_i,\nabla_{e_{\alpha}}e_{\alpha}\right\rangle}-\sum_{i=1}^n{H^{^i}_{\mathcal{F}^{\bot}}([e_{\alpha},e_{i}]^{\bot},e_{\alpha})} +\sum_{i=1}^{n}{\left\langle \nabla_{e_i}\nabla_{e_{\alpha}}e_{i},e_{\alpha}\right\rangle}
\end{eqnarray}

The equality $e_{\alpha}\left\langle e_i,e_{\alpha}\right\rangle=0$ implies $ \left\langle \nabla_{e_{\alpha}}e_i,e_{\alpha}\right\rangle+\left\langle e_i,\nabla_{e_{\alpha}}e_{\alpha}\right\rangle=0$ and then
\begin{eqnarray}
	\left\langle \nabla_{e_i}\nabla_{e_{\alpha}}e_i,e_{\alpha}\right\rangle=-\left\langle \nabla_{e_{\alpha}}e_i,\nabla_{e_i}e_{\alpha}\right\rangle-\left\langle \nabla_{e_i}e_i,\nabla_{e_{\alpha}}e_{\alpha}\right\rangle-\left\langle e_i,\nabla_{e_i}\nabla_{e_{\alpha}}e_{\alpha}\right\rangle
\end{eqnarray}
Finally, using (6) and (7)
\begin{eqnarray*}
e_{\alpha}\left\langle h,e_{\alpha}\right\rangle-\left\|H_{\mathcal{F}}^{\alpha}\right\|^2-\sum\limits_i{R(e_{\alpha},e_i,e_i,e_{\alpha})}\\
	\\
& \hspace{-12cm}	= & \hspace{-6cm}\sum_{i=1}^n{\left\langle \nabla_{e_i}e_i,\nabla_{e_{\alpha}}e_{\alpha}\right\rangle}- \sum\limits_i{H^{i}_{\mathcal{F}^{\bot}}([e_{\alpha},e_{i}]^{\bot},e_{\alpha})} + \sum\limits_i{\left\langle \nabla_{e_i}\nabla_{e_{\alpha}}e_{i},e_{\alpha}\right\rangle}\\
	\\
&	\hspace{-12cm} = & \hspace{-6cm} \sum_{i=1}^n{\left\langle \nabla_{e_i}e_i,\nabla_{e_{\alpha}}e_{\alpha}\right\rangle}- \sum\limits_i{H^{i}_{\mathcal{F}^{\bot}}([e_{\alpha},e_{i}]^{\bot},e_{\alpha})}\\
	\\
&	\hspace{-12cm} - & \hspace{-6cm} \sum\limits_i{(\left\langle \nabla_{e_{\alpha}}e_i,\nabla_{e_i}e_{\alpha}\right\rangle + \left\langle \nabla_{e_i}e_i,\nabla_{e_{\alpha}}e_{\alpha}\right\rangle+\left\langle e_i,\nabla_{e_i}\nabla_{e_{\alpha}}e_{\alpha}\right\rangle)}\\
	\\
&	\hspace{-12cm} = & \hspace{-6cm} \sum_{i=1}^n{\left\langle \nabla_{e_i}e_i,\nabla_{e_{\alpha}}e_{\alpha}\right\rangle}- \sum\limits_i{H^{i}_{\mathcal{F}^{\bot}}([e_{\alpha},e_{i}]^{\bot},e_{\alpha})} -\sum_{i=1}^n{\left\langle \nabla_{e_i}e_i,\nabla_{e_{\alpha}}e_{\alpha}\right\rangle}\\
	\\
& \hspace{-12cm} - & \hspace{-6cm}\sum\limits_{i}{\left\langle \nabla_{e_{\alpha}}e_i,\nabla_{e_{i}}e_{\alpha}\right\rangle}-\sum\limits_i{\left\langle e_i,\nabla_{e_i}\nabla_{e_{\alpha}}e_{\alpha}\right\rangle}\\
	\\
&	\hspace{-12cm} = & \hspace{-6cm} - \sum\limits_i{H^{i}_{\mathcal{F}^{\bot}}([e_{\alpha},e_{i}]^{\bot},e_{\alpha})} - \sum\limits_{i}{\left\langle \nabla_{e_{\alpha}}e_i,\nabla_{e_{i}}e_{\alpha}\right\rangle}-\mathrm{div}_{\mathcal{F}}(\nabla_{e_{\alpha}}e_{\alpha})\\
	\\
& \hspace{-12cm}	= & \hspace{-6cm} - \sum\limits_i{H^{i}_{\mathcal{F}^{\bot}}([e_{\alpha},e_{i}]^{\bot},e_{\alpha})} - \sum\limits_{i}\sum\limits_{\beta}{\left\langle \nabla_{e_{\alpha}}e_i,e_{\beta}\right\rangle\left\langle \nabla_{e_{i}}e_{\alpha},e_{\beta}\right\rangle}-\mathrm{div}_{\mathcal{F}}(\nabla_{e_{\alpha}}e_{\alpha})\\
	\\
&	\hspace{-12cm} = & \hspace{-6cm} - \sum\limits_i{H^{i}_{\mathcal{F}^{\bot}}([e_{\alpha},e_{i}]^{\bot},e_{\alpha})} + \sum\limits_{i}\sum\limits_{\beta}{H^{i}_{\mathcal{F}^{\bot}}(e_{\alpha},e_{\beta})\left\langle \nabla_{e_{i}}e_{\alpha},e_{\beta}\right\rangle}-\mathrm{div}_{\mathcal{F}}(\nabla_{e_{\alpha}}e_{\alpha})\\
	\\
&	\hspace{-12cm} = & \hspace{-6cm} - \sum\limits_i{H^{i}_{\mathcal{F}^{\bot}}([e_{\alpha},e_{i}]^{\bot},e_{\alpha})} + \sum\limits_{i}\sum\limits_{\beta}{H^{i}_{\mathcal{F}^{\bot}}(e_{\alpha},\left\langle \nabla_{e_{i}}e_{\alpha},e_{\beta}\right\rangle e_{\beta})}-\mathrm{div}_{\mathcal{F}}(\nabla_{e_{\alpha}}e_{\alpha})\\
	\\
&	\hspace{-12cm} = & \hspace{-6cm} - \sum\limits_i{H^{i}_{\mathcal{F}^{\bot}}([e_{\alpha},e_{i}]^{\bot},e_{\alpha})} + \sum\limits_{i}{H^{i}_{\mathcal{F}^{\bot}}(e_{\alpha},\sum\limits_{\beta}{\left\langle \nabla_{e_{i}}e_{\alpha},e_{\beta}\right\rangle e_{\beta})}}-\mathrm{div}_{\mathcal{F}}(\nabla_{e_{\alpha}}e_{\alpha})\\
	\\
&	\hspace{-12cm} = &\hspace{-6cm}  - \sum\limits_i{H^{i}_{\mathcal{F}^{\bot}}([e_{\alpha},e_{i}]^{\bot},e_{\alpha})} + \sum\limits_{i}{H^{i}_{\mathcal{F}^{\bot}}(e_{\alpha},\nabla_{e_{i}}^{\bot}e_{\alpha}})-\mathrm{div}_{\mathcal{F}}(\nabla_{e_{\alpha}}e_{\alpha})\\
	\\
& \hspace{-12cm} = & \hspace{-6cm} \sum\limits_{i}{H^{i}_{\mathcal{F}^{\bot}}(e_{\alpha},\nabla_{e_{i}}^{\bot}e_{\alpha}}-[e_{\alpha},e_{i}]^{\bot})-\mathrm{div}_{\mathcal{F}}(\nabla_{e_{\alpha}}e_{\alpha}) \ \ \ \ \ \ \ \ \square
\end{eqnarray*}

\end{proof}

\section{The totally umbilical case}
\begin{prop}
Let $\mathcal{F}$ and $\mathcal{F}^{\bot}$ be two orthogonal foliations of complementary dimensions over a riemannian manifold $M$. If we suppose that $\mathcal{F}^{\bot}$ is totally umbilical, then there exists a basis $\left\{e_{n+1},\ldots,e_{n+p}\right\}$ of the bundle $T\mathcal{F}^{\bot}$ that simultaneously diagonalizes all symmetric operators $A_{e_{i}}$ for $i\in \left\{1,\ldots,n\right\}$.
\end{prop}
\begin{proof}
In fact, there exists a basis of $T\mathcal{F}^{\bot}$ that diagonalizes the operator $A_{e_{1}}$ and then, by hypothesis, $A_{e_{1}}$ is a multiple of the identity matrix. Thus, in this basis, the operator $A_{e_{1}}$ commutes with the operator $A_{e_{2}}$. By a well known theorem of Linear Algebra, there exists a basis of $T\mathcal{F}^{\bot}$ that simultaneously diagonalizes the operators $A_{e_{1}}$ and $A_{e_{2}}$ and then, by hypothesis, both operators are multiples of the identity matrix. By induction, assume that the operators $A_{e_{1}},\ldots,A_{e_{n-1}}$ are multiples of the identity matrix and use the same argument to conclude that there exists a basis of $T\mathcal{F}^{\bot}$ with the desired property $\ \ \ \square$
\end{proof}
\begin{rem}
By proposition 1, for each $i$, there exists a basis $(e_{\alpha})$ of the bundle $T\mathcal{F}^{\bot}$ and a common eigenvalue $\lambda^{i}$ such that
	\begin{eqnarray*}
		A_{e_{i}}(e_{\alpha})=\lambda^{i}e_{\alpha} \ \ \ and \ \ \  H_{\mathcal{F}^{\bot}}^{i}(e_{\alpha},e_{\beta})=\lambda^{i}\delta_{\alpha\beta} 
	\end{eqnarray*}
for all $\alpha, \beta$.
\end{rem}

\begin{theorem}
	If $\mathcal{F}^{\bot}$ is a totally umbilical foliation and $e_{n+1},\ldots,e_{n+p}$ is the basis of the proposition 1, then
	\begin{eqnarray*}
		e_{\alpha}\left\langle h,e_{\alpha}\right\rangle-\left\|H_{_\mathcal{F}}^{^\alpha}\right\|^2-\sum_{i=1}^{n}{R(e_{\alpha},e_i,e_i,e_{\alpha})}=\sum\limits_{i=1}^{n}(\lambda^{i})^{2}-\mathrm{div}_{\mathcal{F}}(\nabla_{e_{\alpha}}e_{\alpha})
	\end{eqnarray*}
\end{theorem}
\begin{proof}
	It is sufficient to show that 
	\begin{eqnarray}
		\sum\limits_{i}{H^{i}_{\mathcal{F}^{\bot}}(e_{\alpha},\nabla_{e_{i}}^{\bot}e_{\alpha}}-[e_{\alpha},e_{i}]^{\bot})=\sum_{i}{(\lambda^{i})^2}
	\end{eqnarray}
Then, by theorem 1, we conclude that
\begin{eqnarray*}
		e_{\alpha}\left\langle h,e_{\alpha}\right\rangle-\left\|H_{_\mathcal{F}}^{^\alpha}\right\|^2-\sum_{i=1}^{n}{R(e_{\alpha},e_i,e_i,e_{\alpha})}=\sum\limits_{i=1}^{n}(\lambda^{i})^{2}-\mathrm{div}_{\mathcal{F}}(\nabla_{e_{\alpha}}e_{\alpha})
	\end{eqnarray*}
The proof of (8) is an easy computation
	\begin{eqnarray*}
\sum_{i}{H^{i}_{\mathcal{F}^{\bot}}(e_{\alpha},\nabla_{e_{i}}^{\bot}e_{\alpha}}\!\!\!\!\!&-&\!\!\!\!\![e_{\alpha},e_{i}]^{\bot})=\sum_{i}{\langle -\nabla_{e_{\alpha}}e_i,\sum_{\beta}{\langle \nabla_{e_{i}}^{\bot}e_{\alpha}-[e_{\alpha},e_{i}]^{\bot},e_{\beta}\rangle e_{\beta}}\rangle}\\
& = &\!\!\!\! \sum_{i}\sum_{\beta}{\langle \nabla_{e_{i}}^{\bot}e_{\alpha}-[e_{\alpha},e_{i}]^{\bot},e_{\beta}\rangle\langle -\nabla_{e_{\alpha}}e_i, e_{\beta}\rangle}\\
& = &\!\!\!\! \sum_{i}\sum_{\beta}{\langle \nabla_{e_{i}}^{\bot}e_{\alpha}-[e_{\alpha},e_{i}]^{\bot},e_{\beta}\rangle \langle A_{e_i}(e_{\alpha}), e_{\beta}\rangle}\\
& = &\!\!\!\! \sum_{i}\sum_{\beta}{\langle \nabla_{e_{i}}^{\bot}e_{\alpha}-[e_{\alpha},e_{i}]^{\bot},e_{\beta}\rangle \langle \lambda^{i}e_{\alpha}, e_{\beta}\rangle}\\
& = &\!\!\!\! \sum_{i}{\lambda^{i} \langle \nabla_{e_{i}}^{\bot}e_{\alpha}-[e_{\alpha},e_{i}]^{\bot},e_{\alpha}\rangle}\\
& = &\!\!\!\! \sum_{i}{\lambda^{i} \langle \nabla_{e_{i}}^{\bot}e_{\alpha},e_{\alpha}\rangle}-\sum_{i}{\lambda^{i} \langle [e_{\alpha},e_{i}]^{\bot},e_{\alpha}\rangle}\\
& = &\!\!\!\! \sum_{i}{\lambda^{i} \langle \nabla_{e_{i}}^{\bot}e_{\alpha},e_{\alpha}\rangle}-\sum_{i}{\lambda^{i} \langle \nabla_{e_{\alpha}}^{\bot}e_i-\nabla_{e_i}^{\bot}e_{\alpha},e_{\alpha}\rangle}\\
& = &\!\!\!\! 2\sum_{i}{\lambda^{i} \langle \nabla_{e_{i}}^{\bot}e_{\alpha},e_{\alpha}\rangle}-\sum_{i}{\lambda^{i} \langle \nabla_{e_{\alpha}}^{\bot}e_i,e_{\alpha}\rangle}\\
& = &\!\!\!\! 2\sum_{i}{\lambda^{i} \langle \nabla_{e_{i}}e_{\alpha},e_{\alpha}\rangle}-\sum_{i}{\lambda^{i} \langle \nabla_{e_{\alpha}}e_i,e_{\alpha}\rangle}\\
& = &\!\!\!\! 0+\sum_{i}{\lambda^{i} \langle -\nabla_{e_{\alpha}}e_i,e_{\alpha}\rangle}\\
& = &\!\!\!\! \sum_{i}{\lambda^{i} \langle A_{e_i}(e_{\alpha}),e_{\alpha}\rangle}\\
& = &\!\!\!\! \sum_{i}{\lambda^{i} \langle \lambda^{i}e_{\alpha},e_{\alpha}\rangle}\\
& = &\!\!\!\! \sum_{i}{(\lambda^{i})^2} \ \ \ \ \ \ \square
\end{eqnarray*}
\end{proof}
\par In this section we suppose that $h$ and $h^{\bot}$ are normalized, that is, we have the following equalities
\begin{eqnarray*}
	h\!\!\!&=&\!\!\!\frac{1}{n}\sum\limits_{i}\left(\nabla_{e_{i}}e_{i}\right)^{\bot}\\
	h^{\bot}\!\!\!&=&\!\!\!\frac{1}{p}\sum\limits_{\alpha}\left(\nabla_{e_{\alpha}}e_{\alpha}\right)^{\top}
\end{eqnarray*}
 If the foliation $\mathcal{F}^{\bot}$ is totally umbilical, then
 \begin{eqnarray}
	 h^{\bot}=\sum_{i=1}^{n}{\lambda(x,e_i)e_i}=\sum_{i=1}^{n}{\lambda^{i}e_i}
 \end{eqnarray}
 where $\lambda^{i}$ is the eingevalue of the Weingarten operator $A_{e_{i}}$.
\begin{prop}
	Let $\mathcal{F}$ and $\mathcal{F}^{\bot}$ be two orthogonal foliations of complementary dimensions over a riemannian manifold $M$. If we suppose that $\mathcal{F}^{\bot}$ is totally umbilical, then
	\begin{eqnarray*}
		n\left\langle \nabla_{e_{\alpha}}h,e_{\alpha}\right\rangle-\left\|H_{_\mathcal{F}}^{^\alpha}\right\|^2-\sum\limits_{i=1}^{n}{R(e_{\alpha},e_i,e_i,e_{\alpha})}-\|h^{\bot}\|^{2}+\mathrm{div}_{\mathcal{F}}(h^{\bot})=0.
	\end{eqnarray*}
\end{prop}
\begin{proof}
By (8) and (9) we have 
\begin{eqnarray}
\sum_{i}{H^{i}_{\mathcal{F}^{\bot}}(e_{\alpha},\nabla_{e_{i}}^{\bot}e_{\alpha}}-[e_{\alpha},e_{i}]^{\bot}) \!\!\!& = &\!\!\! \|h^{\bot}\|^{2}
\end{eqnarray}
We also have the following equation
\begin{eqnarray}
\mathrm{div}_{\mathcal{F}}(\nabla_{e_{\alpha}}e_{\alpha}) \!\!\!& = &\!\!\! \mathrm{div}_{\mathcal{F}}(h^{\bot})-n\left\langle \nabla_{e_{\alpha}}e_{\alpha},h\right\rangle.
\end{eqnarray}
In fact, by (9) 
\begin{eqnarray*}
	h^{\bot}=\sum_{i=1}^{n}{\lambda^{i}e_i}=\sum\limits_{i=1}^{n}\left\langle \nabla_{e_{\alpha}}e_{\alpha},e_{i}\right\rangle e_{i}=\left(\nabla_{e_{\alpha}}e_{\alpha}\right)^{\top}
\end{eqnarray*}
and then
\begin{eqnarray*}
\mathrm{div}_{\mathcal{F}}(\nabla_{e_{\alpha}}e_{\alpha}) \!\!\!& = &\!\!\! \sum_{i=1}^{n}{\left\langle \nabla_{e_i}\nabla_{e_{\alpha}}e_{\alpha},e_i\right\rangle}\\
& = &\!\!\! \sum_{i=1}^{n}{e_i\left\langle \nabla_{e_{\alpha}}e_{\alpha},e_i\right\rangle}-\sum_{i=1}^{n}{\left\langle \nabla_{e_{\alpha}}e_{\alpha},\nabla_{e_{i}}e_{i}\right\rangle}\\
& = &\!\!\! \sum_{i=1}^{n}{e_i\left\langle \nabla_{e_{\alpha}}e_{\alpha},e_i\right\rangle}-\sum_{i=1}^{n}{\left\langle \nabla_{e_{\alpha}}e_{\alpha},\nabla^{\top}_{e_{i}}e_{i}\right\rangle}-\sum_{i=1}^{n}{\left\langle \nabla_{e_{\alpha}}e_{\alpha},\nabla^{\bot}_{e_{i}}e_{i}\right\rangle}\\
& = &\!\!\! \sum_{i=1}^{n}{e_i\left\langle \nabla_{e_{\alpha}}e_{\alpha},e_i\right\rangle}-n\left\langle \nabla_{e_{\alpha}}e_{\alpha},h\right\rangle-\sum_{i=1}^{n}{\left\langle \nabla_{e_{\alpha}}e_{\alpha},\nabla^{\top}_{e_{i}}e_{i}\right\rangle}\\
& = & \!\!\!\sum_{i=1}^{n}{e_i\left\langle h^{\bot},e_i\right\rangle}-n\left\langle \nabla_{e_{\alpha}}e_{\alpha},h\right\rangle-\sum_{i=1}^{n}{\left\langle h^{\bot},\nabla^{\top}_{e_{i}}e_{i}\right\rangle}\\
& = &\!\!\! \sum_{i=1}^{n}{\left\langle \nabla_{e_i}h^{\bot},e_i\right\rangle}-n\left\langle \nabla_{e_{\alpha}}e_{\alpha},h\right\rangle\\
& = & \!\!\!\mathrm{div}_{\mathcal{F}}(h^{\bot})-n\left\langle \nabla_{e_{\alpha}}e_{\alpha},h\right\rangle
\end{eqnarray*}
We conclude the proof using (10), (11) and the theorem 2 $\ \ \square$
\end{proof}
\par Using the proposition 2 we obtain the following theorem

\begin{theorem}
	Let $M$ be a complete riemannian manifold and denote by $\mathcal{F}$ and $\mathcal{F}^{\bot}$ two orthogonal foliations of complementary dimensions on $M$. Suppose that
	\begin{enumerate}
		\item $\mathcal{F}^{\bot}$ is totally umbilical for $p\geq 2$ or a foliation by round circles
		\item $h^{\bot}$ has bounded length and $\left\|h^{\bot}\right\|$ attains a maximum at a leaf $L^{\bot}$ of $\mathcal{F}^{\bot}$
		\item $\mathrm{Tr}(K_{ij}^{\alpha})\geq 0$
		\item $\mathrm{div}_{\mathcal{F}}(h^{\bot})(x)\leq\mathrm{Tr}(K_{ij}^{\alpha})(x)$ for all $x\in L^{\bot}$
	\end{enumerate}
	Then $\mathrm{Tr}(K_{ij}^{\alpha})=0$, $\mathcal{F}$ and $\mathcal{F}^{\bot}$ are totally geodesic and $M$ is locally a riemannian product of a leave of $\mathcal{F}$ and a leave of $\mathcal{F}^{\bot}$.
\end{theorem}
\par Then, as a corollary, we have the following theorem (proved in \cite{Almeida2017573})
\begin{theorem}
Let $M$ be a complete riemannian manifold of constant curvature $c\geq 0$ and denote by $\mathcal{F}$ and $\mathcal{F}^{\bot}$ two orthogonal foliations of complementary dimensions on $M$. Suppose that
	\begin{enumerate}
		\item $\mathcal{F}^{\bot}$ is totally umbilical for $p\geq 2$ or a foliation by round circles
		\item $h^{\bot}$ has bounded length and $\left\|h^{\bot}\right\|$ attains a maximum at a leaf $L^{\bot}$ of $\mathcal{F}^{\bot}$
		\item $\mathrm{div}_{\mathcal{F}}(h^{\bot})(x)\leq nc$ for all $x\in L^{\bot}$
	\end{enumerate}
	Then $c=0$, $\mathcal{F}$ and $\mathcal{F}^{\bot}$ are totally geodesic and $M$ is locally a riemannian product of a leave of $\mathcal{F}$ and a leave of $\mathcal{F}^{\bot}$.
\end{theorem}
\par Our proof of the theorem 3 is very similar to the proof of theorem 4, but we obtain a more general result by applying the proposition 2. The following lemma is an example of application of the proposition 2, it is a little bit more general than the lemma 4.1 of \cite{Almeida2017573}.
\begin{lemma}
	$h$ vanishes along a leaf $L^{\bot}$ of $\mathcal{F}^{\bot}$ where $\left\|h^{\bot}\right\|$ attains a maximum. Furthermore, the mean curvature $h^{\bot}$ of each leaf of $\mathcal{F}^{\bot}$ vanishes, $\mathrm{Tr}(K_{ij}^{\alpha})=0$ and $h=0$.
\end{lemma}
\begin{proof}
	Suppose that $h(p)\neq 0$ for some point $p\in L^{\bot}$. With this assumption, let $\gamma:I\longrightarrow L^{\bot}$ be a maximal integral curve of $h$ with $I=[0,b)$ and $\gamma(0)=p$. In a small neighborhood of $p$ we define $e_{\alpha}=h/\left\|h\right\|$. The proposition 2 and the hypothesis implies the following inequality
	\begin{eqnarray*}
		ne_{\alpha}(\left\|h\right\|)=n\left\langle \nabla_{e_{\alpha}}h,e_{\alpha}\right\rangle=\left\|H_{\mathcal{F}}^{\alpha}\right\|^{2}+\mathrm{Tr}(K_{ij}^{\alpha})+\left\|h^{\bot}\right\|^{2}-\mathrm{div}_{\mathcal{F}}(h^{\bot})\geq \left\|H_{\mathcal{F}}^{\alpha}\right\|^{2}
	\end{eqnarray*}
and by the Cauchy-Schwarz inequality we have
	\begin{eqnarray*}
		ne_{\alpha}(\left\|h\right\|)\geq \left\|H_{\mathcal{F}}^{\alpha}\right\|^{2}\geq \left\|h\right\|^{2}/n
	\end{eqnarray*}
therefore
\begin{eqnarray*}
	\left\|h\circ \gamma\right\|'=\left(\left\|h\right\|\circ\gamma\right)'=\gamma'(\left\|h\right\|)=\left\|h\right\|e_{\alpha}(\left\|h\right\|)\geq \left\|h\circ\gamma\right\|^{3}/n^{2}
\end{eqnarray*}
Note that $\gamma$ has infinite length because $\left\|h\right\|$ is increasing along the curve $\gamma$. Then, we can now consider the function $g:[0,+\infty[\longrightarrow\mathbb{R}$ given by
\begin{eqnarray*}
	g=-\frac{1}{\left\|h\circ \gamma_{1}\right\|}
\end{eqnarray*}
 where $\gamma_{1}$ is a reparametrization of $\gamma$ by arc length. Note that
 \begin{eqnarray*}
	 g'=\left\|h\circ\gamma_{1}\right\|'/\left\|h\circ\gamma\right\|^{2}\geq \left\|h\circ\gamma_{1}\right\|/n^{2}\geq\left\|h(p)\right\|/n^{2}
 \end{eqnarray*}
Fixing a positive number $a$, we have 
\begin{eqnarray*}
	\frac{1}{\left\|h\circ \gamma_{1}(a)\right\|}\geq -\frac{1}{\left\|h\circ\gamma_{2}(s)\right\|}+\frac{1}{\left\|h\circ\gamma_{1}(a)\right\|}=g(s)-g(a)\geq \frac{\left\|h(p)\right\|}{n^{2}}(s-a)
\end{eqnarray*}
for all $s>a$ and this is impossible. The conclusion is that $h=0$ on $L^{\bot}$. As a consequence, the mean curvature $h^{\bot}$ of each leaf of $\mathcal{F}^{\bot}$ vanishes, $\mathrm{Tr}(K_{ij}^{\alpha})=0$ and $h=0$. In fact, again by proposition 2, on the leaf $L^{\bot}$ we have
\begin{eqnarray*}
	0=n\left\langle \nabla_{e_{\alpha}}h,e_{\alpha}\right\rangle=\left\|H_{\mathcal{F}}^{\alpha}\right\|^{2}+\mathrm{Tr}(K_{ij}^{\alpha})+\left\|h^{\bot}\right\|^{2}-\mathrm{div}_{\mathcal{F}}(h^{\bot})\geq \left\|h^{\bot}\right\|^{2}
\end{eqnarray*}
It follows that $\left\|h^{\bot}\right\|=0$ on the leaf $L^{\bot}$ where $\left\|h^{\bot}\right\|$ attains a maximum. Then $h^{\bot}=0$ on $M$. Since $\mathrm{div}(h^{\bot})=0$, it also follows that $\mathrm{Tr}(K_{ij})=0$. Finally, let $q$ any point on $M$ and let $L_{q}^{\bot}$ be a leaf of $\mathcal{F}^{\bot}$ passing through the point $q$. We already know that $h=0$ on the leaf $L_{q}^{\bot}$ and then, in particular, $h(q)=0$ $\ \ \square$
\end{proof}
\par The proof of the theorem 3 now follows from the fact that $H_{\mathcal{F}}^{\alpha}=0$ (by proposition 2 and the lemma 1). Then, as a consequence, $\mathcal{F}$ is a totally geodesic foliation. As $\mathcal{F}^{\bot}$ is a totally umbilical foliation and $h^{\bot}=0$, we have that $\mathcal{F}^{\bot}$ is also a totally geodesic foliation.  
\section{The integral formula}
\begin{theorem}
Let $M$ be a closed oriented riemannian manifold and denote by $\mathcal{F}$ and $\mathcal{F}^{\bot}$ two orthogonal foliations of complementary dimensions on $M$. Then 
	\begin{eqnarray*}
\int_{M}\left(e_{\alpha}\left\langle h,e_{\alpha}\right\rangle-\left\|H_{_\mathcal{F}}^{^\alpha}\right\|^2-\mathrm{Tr}(K^{\alpha}_{ij})-\sum\limits_{i=1}^{n}\mathcal{H}_{\bot}^{i,\alpha}-\mathrm{div}_{\mathcal{F}^{\bot}}(\nabla_{e_{\alpha}}e_{\alpha})\right)d\nu=0
\end{eqnarray*}
where $\displaystyle{\mathcal{H}_{\bot}^{i,\alpha}=H_{\mathcal{F^{\bot}}}^{^i}(e_{\alpha},\nabla^{\bot}_{e_{i}}e_{\alpha}-[e_{\alpha},e_{i}]^{\bot})}$.
\end{theorem}
\begin{proof}
The result is a consequence of 
\begin{eqnarray*}
	\mathrm{div}_{\mathcal{F}}(\nabla_{e_{\alpha}}e_{\alpha})=\mathrm{div}(\nabla_{e_{\alpha}}e_{\alpha})-\mathrm{div}_{\mathcal{F}^{\bot}}(\nabla_{e_{\alpha}}e_{\alpha})
\end{eqnarray*}
 and of theorem 1  $\ \square$
\end{proof}
We will now use the following notations and definitions of the reference \cite{Walczak1990}
\begin{eqnarray*}
	B_{\mathcal{F}}(e_{i},e_{j})\!\!\!& = &\!\!\!\left(\nabla_{e_{i}}e_{j}\right)^{\bot}\\
	B_{\mathcal{F}^{\bot}}(e_{\alpha},e_{\beta})\!\!\!& = &\!\!\!\left(\nabla_{e_{\alpha}}e_{\beta}\right)^{\top}\\
	K(\mathcal{F},\mathcal{F}^{\bot})\!\!\!& = &\!\!\!\sum\limits_{i,\alpha}R(e_{\alpha},e_{i},e_{i},e_{\alpha})
\end{eqnarray*}
As a corollary of theorem 5, we have the following theorem (proved in \cite{Walczak1990})
\begin{theorem}
	Let $M$ be a closed oriented riemannian manifold and denote by $\mathcal{F}$ and $\mathcal{F}^{\bot}$ two orthogonal foliations of complementary dimensions on $M$. Then
	\begin{eqnarray*}
	\int_{M}\left(K(\mathcal{F},\mathcal{F}^{\bot})+\|B_{\mathcal{F}}\|^2+\|B_{\mathcal{F}^{\bot}}\|^2-\|h\|^2-\|h^{\bot}\|^2\right)d\nu=0.	
	\end{eqnarray*}
\end{theorem}
To prove the theorem 6, we first observe that
\begin{eqnarray*}
			B_{\mathcal{F}}(e_{i},e_{j})=(\nabla_{e_{i}} e_{j})^{\bot}=\sum\limits_{\alpha}\left\langle \nabla_{e_{i}}e_{j},e_{\alpha}\right\rangle e_{\alpha}=\sum\limits_{\alpha}H_{\mathcal{F}}^{\alpha}(e_{i},e_{j})e_{\alpha}
\end{eqnarray*}
and then we obtain
\begin{eqnarray}
	\left\|B_{\mathcal{F}}\right\|^{2}=\sum\limits_{i,j}\left\langle (\nabla_{e_{i}}e_{j})^{\bot},(\nabla_{e_{j}}e_{i})^{\bot}\right\rangle=\sum\limits_{\alpha}\left\|H_{\mathcal{F}}^{\alpha}\right\|^{2}=\sum\limits_{i,\alpha}(\lambda_{i}^{\alpha})^{2}
\end{eqnarray}
On the other hand, by equation (5) and the definition of $H_{\mathcal{F}^{\bot}}^{i}$, we have
\begin{eqnarray*}
	-H_{\mathcal{F}^{\bot}}^{i}([e_{\alpha},e_{i}]^{\bot},e_{\alpha})\!\!\!&=&\!\!\!\left\langle \nabla_{[e_{\alpha},e_{i}]}e_{i},e_{\alpha}\right\rangle-(\lambda_{i}^{\alpha})^{2}\\
	H_{\mathcal{F}^{\bot}}^{i}(e_{\alpha},\nabla_{e_{i}}^{\bot}e_{\alpha})\!\!\!&=&\!\!\!-\left\langle \nabla_{e_{\alpha}}e_{i},\nabla_{e_{
	i}}^{\bot}e_{\alpha}\right\rangle
\end{eqnarray*}
Adding the above equations we obtain
\begin{eqnarray*}
	H_{\mathcal{F}^{\bot}}^{i}(e_{\alpha},\nabla_{e_{i}}^{\bot}e_{\alpha}-[e_{\alpha},e_{i}]^{\bot} )=-(\left\langle \nabla_{e_{\alpha}}e_{i},\nabla_{e_{i}}e_{\alpha}\right\rangle-\left\langle \nabla_{[e_{\alpha},e_{i}]}e_{i},e_{\alpha}\right\rangle)-(\lambda_{i}^{\alpha})^{2}
\end{eqnarray*}
We observe now that
\begin{eqnarray*}
	\left\langle \nabla_{e_{\alpha}}e_{i},\nabla_{e_{i}}e_{\alpha}\right\rangle\!\!\!&-&\!\!\!\left\langle \nabla_{[e_{\alpha},e_{i}]}e_{i},e_{\alpha}\right\rangle=\left\langle \nabla_{e_{i}}e_{\alpha},\nabla_{e_{\alpha}}e_{i}\right\rangle+\left\langle \nabla_{[e_{\alpha},e_{i}]}e_{\alpha},e_{i}\right\rangle\\
	& = &\!\!\!\left\langle \nabla_{e_{i}}e_{\alpha},\nabla_{e_{\alpha}}e_{i}\right\rangle-\left\langle \nabla_{\nabla_{e_{i}}e_{\alpha}}e_{\alpha},e_{i}\right\rangle+\left\langle \nabla_{\nabla_{e_{\alpha}}e_{i}}e_{\alpha},e_{i}\right\rangle\\
	& = &\!\!\!\sum\limits_{j}\left\langle \nabla_{e_{i}}e_{\alpha},e_{j}\right\rangle\left\langle \nabla_{e_{\alpha}}e_{i},e_{j}\right\rangle+\sum\limits_{\beta}\left\langle\nabla_{e_{i}}e_{\alpha},e_{\beta}\right\rangle\left\langle\nabla_{e_{\alpha}}e_{i},e_{\beta}\right\rangle\\
	& - &\!\!\!\sum\limits_{j}\left\langle \nabla_{e_{i}}e_{\alpha},e_{j}\right\rangle\left\langle \nabla_{e_{j}}e_{\alpha},e_{i}\right\rangle-\sum\limits_{\beta}\left\langle\nabla_{e_{i}}e_{\alpha},e_{\beta}\right\rangle\left\langle\nabla_{e_{\beta}}e_{\alpha},e_{i}\right\rangle\\
	& + &\!\!\!\sum\limits_{j}\left\langle \nabla_{e_{\alpha}}e_{i},e_{j}\right\rangle\left\langle \nabla_{e_{j}}e_{\alpha},e_{i}\right\rangle+\sum\limits_{\beta}\left\langle\nabla_{e_{\alpha}}e_{i},e_{\beta}\right\rangle\left\langle\nabla_{e_{\beta}}e_{\alpha},e_{i}\right\rangle
\end{eqnarray*}
and then
\begin{eqnarray}
	\nonumber\sum\limits_{i,\alpha} \!\!\!\!\!&H_{\mathcal{F^{\bot}}}^{^i}&\!\!\!\!\!(e_{\alpha},\nabla^{\bot}_{e_{i}}e_{\alpha}-[e_{\alpha},e_{i}]^{\bot})\\
	\nonumber&=&\!\!\!\!\!\sum\limits_{i,\alpha}\left[-(\left\langle \nabla_{e_{\alpha}}e_{i},\nabla_{e_{i}}e_{\alpha}\right\rangle-\left\langle \nabla_{[e_{\alpha},e_{i}]}e_{i},e_{\alpha}\right\rangle)-(\lambda_{i}^{\alpha})^{2}\right]\\
	\nonumber&=&\!\!\!-\sum\limits_{i,j,\alpha}[\left\langle \nabla_{e_{i}}e_{\alpha},e_{j}\right\rangle\left\langle \nabla_{e_{\alpha}}e_{i},e_{j}\right\rangle+\left\langle \nabla_{e_{\alpha}}e_{i},e_{j}\right\rangle\left\langle \nabla_{e_{j}}e_{\alpha},e_{i}\right\rangle\\
	\nonumber&-&\!\!\!\!\!\left\langle e_{\alpha},\nabla_{e_{i}}e_{j}\right\rangle\left\langle e_{\alpha},\nabla_{e_{j}}e_{i}\right\rangle]\\
	\nonumber&-&\!\!\!\!\!\sum\limits_{i,\alpha,\beta}[\left\langle\nabla_{e_{i}}e_{\alpha},e_{\beta}\right\rangle\left\langle\nabla_{e_{\alpha}}e_{i},e_{\beta}\right\rangle+\left\langle\nabla_{e_{i}}e_{\alpha},e_{\beta}\right\rangle\left\langle e_{\alpha},\nabla_{e_{\beta}}e_{i}\right\rangle\\
	\nonumber&-&\!\!\!\!\!\left\langle e_{i},\nabla_{e_{\alpha}}e_{\beta}\right\rangle\left\langle\nabla_{e_{\beta}}e_{\alpha},e_{i}\right\rangle]-\sum\limits_{i,\alpha}(\lambda_{i}^{\alpha})^{2}\\
	\nonumber&=&\!\!\!\!\!\sum\limits_{i,j,\alpha}\left\langle e_{\alpha},\nabla_{e_{i}}e_{j}\right\rangle\left\langle e_{\alpha},\nabla_{e_{j}}e_{i}\right\rangle+\sum\limits_{i,\alpha,\beta}\left\langle e_{i},\nabla_{e_{\alpha}}e_{\beta}\right\rangle\left\langle\nabla_{e_{\beta}}e_{\alpha},e_{i}\right\rangle-\sum\limits_{i,\alpha}(\lambda_{i}^{\alpha})^{2}\\
	\nonumber&=&\!\!\!\!\!\sum\limits_{i,j}\left\langle (\nabla_{e_{i}}e_{j})^{\bot},(\nabla_{e_{j}}e_{i})^{\bot}\right\rangle+\sum\limits_{\alpha,\beta}\left\langle (\nabla_{e_{\alpha}}e_{\beta})^{\top},(\nabla_{e_{\beta}}e_{\alpha})^{\top}\right\rangle-\sum\limits_{i,\alpha}(\lambda_{i}^{\alpha})^{2}\\
	\nonumber&=&\!\!\!\!\!\left\|B_{\mathcal{F}}\right\|^{2}+\left\|B_{\mathcal{F}^{\bot}}\right\|^{2}-\sum\limits_{i,\alpha}(\lambda_{i}^{\alpha})^{2}\\
	&=&\!\!\!\!\!\left\|B_{\mathcal{F}^{\bot}}\right\|^{2}
\end{eqnarray}
\begin{lemma}
	We have the folowing equality
	\begin{eqnarray*}
\sum\limits_{\alpha}\mathrm{div}_{\mathcal{F}}(\nabla_{e_{\alpha}}e_{\alpha})\!\!\!&=&\!\!\!\mathrm{div}_{\mathcal{F}}(h^{\bot})-\sum\limits_{\alpha}\left\langle h,\nabla_{e_{\alpha}}e_{\alpha}\right\rangle
\end{eqnarray*}
\end{lemma}
\begin{proof}
By definition
\begin{eqnarray*}
	&\mathrm{div}_{\mathcal{F}}&\!\!\!\!\!(\nabla_{e_{\alpha}}e_{\alpha})=\sum\limits_{i}\left\langle e_{i},\nabla_{e_{i}}(\nabla_{\alpha}e_{\alpha})\right\rangle\\
	&=&\!\!\!\!\!\sum\limits_{i}\left[e_{i}\left\langle e_{i},\nabla_{e_{\alpha}}e_{\alpha}\right\rangle-\left\langle \nabla_{e_{i}}e_{i},\nabla_{e_{\alpha}}e_{\alpha}\right\rangle\right]\\
	&=&\!\!\!\!\!\sum\limits_{i}e_{i}\left\langle e_{i},\nabla_{e_{\alpha}}e_{\alpha}\right\rangle-\sum\limits_{i}\left\langle \sum\limits_{j}\left\langle \nabla_{e_{i}}e_{i},e_{j}\right\rangle e_{j}+\sum\limits_{\beta}\left\langle \nabla_{e_{i}}e_{i},e_{\beta}\right\rangle e_{\beta},\nabla_{e_{\alpha}}e_{\alpha}\right\rangle\\
	&=&\!\!\!\!\!\sum\limits_{i}e_{i}\left\langle e_{i},\nabla_{e_{\alpha}}e_{\alpha}\right\rangle-\sum\limits_{j}\left\langle \nabla_{\alpha}e_{\alpha},e_{j}\right\rangle\left\langle \sum\limits_{i}\nabla_{e_{i}}e_{i},e_{j}\right\rangle-\left\langle h,\nabla_{e_{\alpha}}e_{\alpha}\right\rangle\\
	&=&\!\!\!\!\!\sum\limits_{i}e_{i}\left\langle e_{i},\nabla_{e_{\alpha}}e_{\alpha}\right\rangle-\left\langle \sum\limits_{i}\nabla_{e_{i}}e_{i},\sum\limits_{j}\left\langle \nabla_{\alpha}e_{\alpha},e_{j}\right\rangle e_{j}\right\rangle-\left\langle h,\nabla_{e_{\alpha}}e_{\alpha}\right\rangle
\end{eqnarray*}
and then 
\begin{eqnarray}
	\sum\limits_{\alpha}\mathrm{div}_{\mathcal{F}}(\nabla_{e_{\alpha}}e_{\alpha})=\sum\limits_{i,\alpha}e_{i}\left\langle e_{i},\nabla_{e_{\alpha}}e_{\alpha}\right\rangle-\left\langle h^{\bot}, \sum\limits_{i}\nabla_{e_{i}}e_{i}\right\rangle-\sum\limits_{\alpha}\left\langle h,\nabla_{e_{\alpha}}e_{\alpha}\right\rangle
\end{eqnarray}
Analogously, we have
\begin{eqnarray}
	\mathrm{div}_{\mathcal{F}}(h^{\bot})=\sum\limits_{i,\alpha}e_{i}\left\langle e_{i},\nabla_{e_{\alpha}}e_{\alpha}\right\rangle-\left\langle h^{\bot}, \sum\limits_{i}\nabla_{e_{i}}e_{i}\right\rangle
\end{eqnarray}
Substituting (15) in (14) we obtain the desired equality $\ \ \square$
\end{proof}
\par Now, we proceed the proof of the theorem 6. The equation of theorem 1 is
\begin{eqnarray*}
e_{\alpha}\left\langle h,e_{\alpha}\right\rangle-\left\|H_{_\mathcal{F}}^{^\alpha}\right\|^2-\mathrm{Tr}(K^{\alpha}_{ij})=\sum\limits_{i=1}^{n}H_{\mathcal{F^{\bot}}}^{^i}(e_{\alpha},\nabla^{\bot}_{e_{i}}e_{\alpha}-[e_{\alpha},e_{i}]^{\bot})-\mathrm{div}_{\mathcal{F}}(\nabla_{e_{\alpha}}e_{\alpha})
\end{eqnarray*}
Summing (in $\alpha$) all these equations, using the relations (12)-(13) and also the definiton of $K=K(\mathcal{F},\mathcal{F}^{\bot})$, we can write
\begin{eqnarray*}
\sum\limits_{\alpha} e_{\alpha}\left\langle h,e_{\alpha}\right\rangle-\left\|B_{\mathcal{F}}\right\|^{2}-K=\left\|B_{\mathcal{F}^{\bot}}\right\|^{2}-\sum\limits_{\alpha}\mathrm{div}_{\mathcal{F}}(\nabla_{e_{\alpha}}e_{\alpha})
\end{eqnarray*}
Using now the following elementary equations (see \cite{Walczak1990})
\begin{eqnarray*}
	\mathrm{div}_{\mathcal{F}^{\bot}}(h)=\mathrm{div}(h)+\left\|h\right\|^{2} \ \ \ \mathrm{and} \ \ \ \mathrm{div}_{\mathcal{F}}(h^{\bot})=\mathrm{div}(h^{\bot})+\left\|h^{\bot}\right\|^{2}
\end{eqnarray*}
 and using the lemma 2, we obtain  
\begin{eqnarray*}
	\sum\limits_{\alpha} e_{\alpha}\left\langle h,e_{\alpha}\right\rangle\!\!\!&-&\!\!\!\left\|B_{\mathcal{F}}\right\|^{2}-K=\left\|B_{\mathcal{F}^{\bot}}\right\|^{2}-\left(\mathrm{div}_{\mathcal{F}}(h^{\bot})-\sum\limits_{\alpha}\left\langle h,\nabla_{e_{\alpha}}e_{\alpha}\right\rangle\right)\\
	&=&\!\!\!\left\||B_{\mathcal{F}^{\bot}}\right\|^{2}-\mathrm{div}_{\mathcal{F}}(h^{\bot})+\sum\limits_{\alpha}\left\langle h,\nabla_{e_{\alpha}}e_{\alpha}\right\rangle\\
	&=&\!\!\!\left\|B_{\mathcal{F}^{\bot}}\right\|^{2}-\mathrm{div}_{\mathcal{F}}(h^{\bot})+\sum\limits_{\alpha}e_{\alpha}\left\langle h,e_{\alpha}\right\rangle-\sum\limits_{\alpha}\left\langle e_{\alpha},\nabla_{e_{\alpha}}h\right\rangle\\
	&=&\!\!\!\left\|B_{\mathcal{F}^{\bot}}\right\|^{2}-\mathrm{div}_{\mathcal{F}}(h^{\bot})+\sum\limits_{\alpha}e_{\alpha}\left\langle h,e_{\alpha}\right\rangle-\mathrm{div}_{\mathcal{F}^{\bot}}(h)\\
	&=&\!\!\!\left\|B_{\mathcal{F}^{\bot}}\right\|^{2}-\mathrm{div}_{\mathcal{F}}(h^{\bot})+\sum\limits_{\alpha}e_{\alpha}\left\langle h,e_{\alpha}\right\rangle-\mathrm{div}(h)-\left\|h\right\|^{2}\\
	&=&\!\!\!\left\|B_{\mathcal{F}^{\bot}}\right\|^{2}-\mathrm{div}(h^{\bot})-\left\|h^{\bot}\right\|+\sum\limits_{\alpha}e_{\alpha}\left\langle h,e_{\alpha}\right\rangle-\mathrm{div}(h)-\left\|h\right\|^{2}
\end{eqnarray*}
and we conclude that
\begin{eqnarray*}
	-\left\|B_{\mathcal{F}}\right\|^{2}-K=\left\|B_{\mathcal{F}^{\bot}}\right\|^{2}-\mathrm{div}(h^{\bot})-\left\|h^{\bot}\right\|-\mathrm{div}(h)-\left\|h\right\|^{2}
\end{eqnarray*}
\par The theorem 6 will follow from the integration of the above equality on $M$.

\section{Final remarks}
\begin{enumerate}
	\item The theorem 6 was obtained first by Ranjan \cite{Ranjan1986} and after generalized by Walczak \cite{Walczak1990} for the case of two distributions.
	\item It is also possible to study the completeness of the foliations using the theorem 1. We will do it in a forthcoming paper.
\end{enumerate}



\begin{thebibliography}{99}
\bibitem{Brito1985410} F. G. B. Brito and P. Walczak, \textit{Totally geodesic foliations with integrable normal bundles}, Bol. Soc. Bras. Mat. 17 (1985) 41-46.
 
\bibitem{Abe1973425} K. Abe, \textit{Applications of a Riccati type differential equation to Riemannian manifolds with
  totally geodesic distributions}, Tohoku Math. Journ. 24 (1973) 425-440.

\bibitem{Almeida2017573} S. C. Almeida and  F. G. B. Brito and A. G. Colares, \textit{Umbilic foliations with integrable normal bundle}, Bulletin des Sciences Mathématiques 141 (2017) 573-583.

\bibitem{Walczak1990} P. Walczak, \textit{An integral formula for a riemannian manifolds with two orthogonal complementary distributions}, Colloquium Mathematicum 58 (1990) 243-252


\bibitem{Ranjan1986} A. Ranjan, \textit{Structural equations and an integral formula for foliated manifolds}, Geom. Dedicata 20 (1986) 85-91.



\end{thebibliography}
\end{document}